\documentclass[11 pt,a4paper]{amsart}
\usepackage{amsmath,amscd}
\usepackage{amstext}
\usepackage{amssymb,epsfig}
\usepackage{amsthm}
\usepackage{amsfonts}
\usepackage{graphicx}

\renewcommand{\footnote}{\endnote}
\newtheorem{theorem}{Theorem}
\newtheorem{lemma}[theorem]{Lemma}
\newtheorem{proposition}[theorem]{Proposition}
\newtheorem{definition}[theorem]{Definition}
\newtheorem{remark}[theorem]{Remark}
\newtheorem{corollary}[theorem]{Corollary}

\usepackage[utf8]{inputenc}      
\author{Nicolas Bedaride, Pascal Hubert}
\address{Fédération de recherches des unités de mathématiques de Marseille,
Laboratoire d'Analyse Topologie et Probabilités  UMR 6632 , Avenue
Escadrille Normandie Niemen 13397 Marseille
cedex 20, France} \email{nicolas.bedaride@univ-cezanne.fr\\
hubert@cmi.univ-mrs.fr}

\begin{document}
\begin{abstract}
We consider the billiard map in the hypercube of $\mathbb{R}^d$.
We obtain a language by coding the billiard map by the faces of the hypercube. We investigate 
the complexity function of this
language. We prove that $n^{3d-3}$ is the order of magnitude of
the complexity.\\

RESUME:
On consid\`ere l'application du billard dans le cube de $\mathbb{R}^d$. On code cette application par les faces du cube. On obtient un langage, dont on cherche \`a \'evaluer la complexit\'e. On montre que l'ordre de grandeur de cette fonction est $n^{3d-3}$.

\end{abstract}

\title{Billiard complexity in the hypercube}
\maketitle

{\bf Keywords}: Symbolic dynamic, billiard, words, complexity function.

 {\bf Mots-cl\'es}: Dynamique symbolique, billard, mots, complexit\'e.
 
AMS codes: 37A35 ; 37C35;  05A16; 11N37; 28D.
\begin{center}
{\bf Complexit\'e du billard cubique multi-dimensionnel.}
\end{center}

\section{Introduction}

A billiard ball, i.e.\ a point mass, moves inside a polyhedron $P$
with unit speed along a straight line until it reaches the
boundary $\partial{P}$, then it instantaneously changes direction
according to the mirror law, and continues along the new line.

Label the faces of $P$ by symbols from a finite alphabet
$\mathcal{A}$ whose cardinality equals the number of faces of $P$.
Either we consider the set of billiard orbits in a fixed
direction, or we consider all orbits.

In both cases the orbit of a point corresponds to a word in the
alphabet $\mathcal{A}$ and the set of all the words is a language.
We define the complexity of the language, $p(n)$, by the number of
words of length $n$ that appears in this system. We call the
complexity of an infinite trajectory the directional complexity:
it does not depend on the initial point under suitable
hypotheses. We denote it by $p(n,\omega)$ (where $\omega$ is the initial direction of the trajectory), and the other one the
global complexity or to short simply the complexity. How complex is the
game of billiard inside a polygon or a polyhedron ?
The only general result 
about complexity function is
that the billiard in a polygon has zero entropy \cite{Ga.Kr.Tr}
\cite{Ka}, and thus the two complexities grow sub-exponentially.
For the convex polyhedron the same fact is true \cite{moi3}.

It is possible to compute the complexity for rational polygons
(a polygon is rational if the angles between sides are rational
multiples of $\pi$). For the directional complexity the first
result is in the famous paper of Morse and Hedlund \cite{He.Mo},
and it has been generalized to any rational polygon by Hubert
\cite{Hu}. This directional complexity is always linear in $n$.

For the global complexity in the square coded by two letters,
Mignosi found an explicit formula see \cite{Mig}, \cite{Ber.Poc}.
Then Cassaigne, Hubert and Troubetzkoy \cite{Ca.Hu.Tr} proved that
$p(n)/n^{3}$ has a lower and an upper bound, in the case of rational
convex polygons, and the first author generalized this result to
the case of non convex rational polygons \cite{moi1}. Moreover, for
some regular polygons they showed that $p(n)/n^{3}$ has a limit, and
then calculated it. But even for the hexagon we are not able to obtain an
equivalent statement: we must use the result of Masur that gives the order
of magnitude of the number of saddle connections \cite{Ma.90}.

In the polyhedral case much less is known. The directional
complexity, in the case of the cube, has been computed by Arnoux,
Mauduit, Shiokawa
 and Tamura \cite{Ar.Ma.Sh.Ta} and generalized to the hypercube by
Baryshnikov \cite{Ba}; see also \cite{moi6} for a generalization. 
Moreover, in \cite{moi1} the computation has
been done in the case of some right prisms whose bases tile the plane. For
those polyhedra the directional complexity is always quadratic in $n$.  In the current 
article we compute the global complexity for the
hypercube of $\mathbb{R}^d$ coded with $d$ letters.

\begin{theorem}\label{cube}
Let $p(n,d)$ be the complexity of the language associated to the
hypercubic billiard (coded with d letters). Then there exists
$C_1,C_2\in\mathbb{R}_+$ such that
$$C_1n^{3d-3}\leq p(n,d)\leq C_2n^{3d-3}.$$
\end{theorem}
\begin{remark}
In the proof some constants appear in the inequalities. We
denote them all by the same letter $C$.\\
Moreover we will use the term "cube" even if $d$ is greater than
three.
\end{remark}

\subsection{Overview of the proof}
The proof of \cite{Ca.Hu.Tr} is based on the fact that the
complexity is related to the number of generalized diagonals. A
generalized diagonal is an orbit segment which starts and ends on
a vertex of the polygon (or an edge of the polyhedron). If we wish
to apply this technique to the hypercube, however, a generalized
diagonal is not necessarily associated to a single word, so that we
must modify the proof. First we show that the complexity is
related to the number of words that appear in one diagonal, see
Section 3. Next we begin to count the numbers of those words. We split
the estimates between several parts. Section 5 is devoted to obtaining the
upper bound by a general geometric argument. In Section 6 we establish the lower bound by an induction on the dimension $d$.

\section{Background}
\subsection{Billiard}
In this section we recall some definitions: Let $P$ be a
polyhedron, the billiard map is called $T$ and it is defined on a
subset of $\partial{P}\times \mathbb{R}P^{d-1}$. This space is called the
phase space.

$\bullet$ We will call a face of the cube a face of dimension $d-1$.
If we use a face of smaller dimension we will state the
dimension.

We define a partition $\mathcal{P}$ of the phase space on $d$ sets
by the following method: the boundary of $P$ is partitioned into $d$
sets by identifying the parallel faces of the cube. Then we
consider the partition
$\mathcal{P}_n=\displaystyle\bigvee_{i=0}^{n}T^{-i}\mathcal{P}$.
\begin{definition}
The complexity of the billiard map, denoted $p(n,d)$, is the
number of atoms of $\mathcal{P}_n$.
\end{definition}
$\bullet$ The unfolding of a billiard trajectory: Instead of
reflecting the trajectory in the face we reflect the cube and
follow the straight line. Thus we consider the tiling of
$\mathbb{R}^{d}$ by $\mathbb{Z}^{d}$, and the associated partition
into cubes of edges of length one. In the following when we use the term "face" we mean 
a face of one of those cubes.\\

The following lemma is very useful in the following.
\begin{lemma}\label{proj}
Consider an orthogonal projection on a face of the cube. The
orthogonal projection of a billiard map is a billiard map inside a
cube of dimension equal to the dimension of the face.
\end{lemma}
\subsection{Combinatorics}
\begin{definition}
Let $\mathcal{A}$ be a finite set called the alphabet. By a
language $L$ over $\mathcal{A}$ we always mean a factorial
extendable language: a language is a collection of sets
$(L_n)_{n\geq 0}$ where the only element of $L_0$ is the empty
word, and each $L_n$ consists of words of the form $a_1a_2\dots
a_n$ where $a_i\in\mathcal{A}$ and such that for each $v\in L_n$
there exist $a,b\in\mathcal{A}$ with $av,vb\in L_{n+1}$, and for
all $v\in L_{n+1}$ if $v=au=u'b$ with $a,b\in\mathcal{A}$ then
$u,u'\in L_n$.\\
The complexity function $p:\mathbb{N}\rightarrow\mathbb{N}$ is
defined by $p(n)=card(L_n)$.
\end{definition}

 First we recall a well known result of Cassaigne
concerning combinatorics of words \cite{Ca}.
\begin{definition}
Let $\mathcal{L}(n)$ be a factorial extendable language. For any
$n\geq1$ let $s(n)\!:=p(n+1)-p(n)$. For $v \in \mathcal{L}(n)$ let

$$m_{l}(v)= card\{u\in \Sigma,uv\in \mathcal{L}(n+1)\},$$
$$m_{r}(v)= card\{w\in \Sigma,vw\in \mathcal{L}(n+1)\},$$
$$m_{b}(v)= card\{u\in \Sigma, w\in \Sigma, uvw\in, \mathcal{L}(n+2)\}.$$

A word is called right special if $m_{r}(v)\geq 2$. A word is called left special if
$m_{l}(v)\geq 2$ and it is called bispecial if it is right and left special.
Let $\mathcal{BL}(n)$ be the set of the bispecial words.
\end{definition}
Cassaigne \cite{Ca} has shown:
\begin{lemma}\label{julien}
Let $\mathcal{L}$ be a language such that $m_l(v)\geq1,m_r(v)\geq1$
for all words $v\in\mathcal{L}$. Then the complexity satisfies
$$\forall n\geq 1 \quad s(n+1)-s(n)=\sum_{v\in \mathcal{BL}(n)}{i(v)},$$

where $i(v)=m_{b}(v)-m_{r}(v)-m_{l}(v)+1.$
\end{lemma}

For the proof of the lemma we refer to \cite{Ca} or
\cite{Ca.Hu.Tr}.
\begin{remark}
If we code the billiard map by the sequence of faces hit in a
trajectory, and if we associate the same letter to the parallel
faces of the cube the two definitions of complexity coincide, {\it
i.e} $p(n,d)=p(n)$.
\end{remark}

\subsection{Geometry}
We recall the Euler's formula.
\begin{lemma}\label{euler}\cite{Ful}
Let $P$ be a simply connected polyhedron of $\mathbb{R}^d$. Let
$N_i$ be the number of faces of $P$ of dimension $i$, then we have
$$\displaystyle\sum_{i=0}^{d-1}(-1)^iN_i=1-(-1)^d.$$
\end{lemma}
\begin{remark}\label{rem10}
In the following sections, we will use this formula for some algebraic
manifolds of degree 2, but for simplicity we will always mean
a polyhedron and hyperplanes, since the proofs are the same.
\end{remark}
Now we prove the following result.
\begin{lemma}
Suppose $(H_i)_{i\leq n}$ is a sequence of hyperplanes of $\mathbb{R}^x$ and let $(\mathcal{Q}_i)_{i\in I}$ be the
connected components of $\mathbb{R}^x \setminus H_1\cup\dots\cup H_n$. Then there exists $C(x)>0$ such that:
$$card I\leq C(x)n^x.$$
\end{lemma}
\begin{proof}
We will prove the assumption by induction on $x$. The induction hypothesis states that it is true for all $i<x$.\\

The hyperplanes $(H_i)$ induce a cellular decomposition of
$\mathbb{R}^x$. We will denote $N_i$ the number of cells of
dimension $i$ for $0\leq i\leq x$. We remark that $card I=N_x$. We
begin by obtaining an upper bound for $N_i$ for $0\leq i<d$. We
will see later that this is sufficient to finish the proof.

$\bullet$ Computation of $N_0$. We denote
$\mathcal{H}=\{H_1,\dots,H_n\}$, and consider the map
$$\phi_0:\mathcal{H}^x\rightarrow\{vertices\}\cup \emptyset$$
$$\phi_0:\ (H_{i_1},\dots,H_{i_x})\mapsto \begin{cases}H_{i_1}\cap\dots\cap
H_{i_x}\quad \text{If it is a point}\\
\emptyset\quad\text{otherwise} \end{cases}$$

This map is surjective, thus we deduce $N_0\leq n^x$. Hence the
induction is true for $x=1$.

$\bullet$ Let $\mathcal{E}_i$ the set of subspaces of dimension
$i$ which form the cells of dimension $i$ of the cellular
decomposition. We denote $E_i=card( \mathcal{E}_i$).

 Then the map
$$\mathcal{H}^{x-i}\rightarrow \mathcal{E}_i\cup\emptyset$$
$$(H_{j_1},\dots,H_{j_{x-i}})\mapsto \begin{cases}H_{j_1}\cap\dots\cap H_{j_{x-i}}\quad \text{if}
\quad dim H_{j_1}\cap\dots\cap H_{j_{x-i}}=x\\
\emptyset\quad\text{otherwise}
\end{cases}$$

is surjective by definition of the partition. We deduce for all
$i\leq x-1$: $E_{i}\leq n^{x-i}$.\\
Now it remains to know into how many pieces each space of
dimension $i$ is cut. Let $F\in \mathcal{E}_i$ and $H\in\mathcal{H}$, we
have $F\cap H=\begin{cases}F\\or\\ X\end{cases}$, where $X$ is a
subset of codimension 1 contained in $F$.

The hyperplanes which partition $F$ do not contain $F$, thus
their trace on $F$ is of codimension one. Thus the problem is
reduced to compute the number of connected components of the
partition of $F$ by $m\leq n$ hyperplanes.

The induction hypothesis implies that $N_F\leq C(i)n^i$. Then
$$max_{F\in \mathcal{E}_i}E_F\leq c(i)n^i.$$
We deduce
$$N_i\leq n^{x-i}c(i)n^i\leq c(i)n^x.$$

Euler's formula implies
$$N_x\leq1+\sum_{i=0}^{x-1}C(i)n^{x}.$$
$$N_x\leq Cn^x.$$
The induction process has been completed.
\end{proof}
\begin{corollary}\label{comb}
Let $P$ be a polyhedron of $\mathbb{R}^x$, let $(H_i)_{i\leq n}$ be a sequence of hyperplanes, and let $(\mathcal{Q}_i)_{i\in I}$ be the
connected components of $P\setminus H_1\cup\dots\cup
H_n$. Then there exists $C(x,P)>0$ such that:

$$card I\leq C(x,P)n^x.$$
 
\end{corollary}
\begin{proof}
 We can apply the
same proof in the case where $P$ is a polyhedron: it suffices to add the
hyperplanes which form the boundary of $P$. In this case only the
constant $C$ changes.
\end{proof}
\begin{remark}\label{rem}
If we consider algebraic equations of {\it bounded} degree (by
$\delta$), the same proof works since an intersection of such
manifolds has a bounded number of connected components, and since
the Euler characteristic takes a finite number of values (see Remark \ref{rem10}), only
depending on $x$ and $\delta$.
\end{remark}

\subsection{Number theory}
A general reference for this section is \cite{Ha.Wr}.

\begin{definition}
Let $n$ be an integer, the invertible elements of
$\mathbb{Z}/n\mathbb{Z}$  are denoted by
$(\mathbb{Z}/n\mathbb{Z})^{*}$, and the cardinality of this set is
denoted $\phi(n)$, which is called the Euler's function.
\end{definition}

\begin{definition}
The Moebius function $\mu$ is defined by
\begin{equation*}
\begin{split}
\mu(1)&=1,\\
\mu(p_{1}..p_{k})&={(-1)}^k,\quad p_i\in\mathbb{P}\quad
\text{distinct primes},\\
\mu(n)&=0,\quad \text{if n has a square factor}.
\end{split}
\end{equation*}
\end{definition}

The multiplicative functions $\phi,\mu$ are linked by the
following classical property:
\begin{lemma}\label{mob}
For all positive integer $n$  the following holds:
\begin{equation*}
\begin{split}
\sum_{d|n}\mu(d)&=
\begin{cases}
1&\quad\text{if}\quad n=1,\\
0&\quad\text{else},
\end{cases}\\
\frac{\phi(n)}{n}&=\sum_{d|n}\frac{\mu(d)}{d}.
\end{split}
\end{equation*}
\end{lemma}

Now we use the above Lemma to obtain the following result.
\begin{lemma}\label{nombre}
For all integer $p\geq 1$, there exists a $C>0$ such that for all
$n$ the following holds:
$$\sum_{l\leq n}S_l\geq Cn^{p+2},$$
where $S_l=\displaystyle\sum_{\substack{m\leq
l\\\gcd(m,l)=1}}m^{p}$.
\end{lemma}
We give a proof of the Lemma for the sake of the completeness. The integer part
is denoted by $E()$.
\begin{proof}
By Lemma \ref{mob} we have
$$S_l=\sum_{\substack{m\leq
l\\\gcd(m,l)=1}}m^{p}=\sum_{m\leq l}\sum_{d|m,d|l}\mu(d)m^p,$$
$$S_l=\sum_{k\leq l/d}\mu(d)k^pd^p,$$
$$S_l=\sum_{d|l}\mu(d)d^p[C_{p+1}(\frac{l}{d})^{p+1}+C_{p}(\frac{l}{d})^{p}+O(\frac{l}{d})^{p-1}],$$
$$S_l=C_{p+1}l^{p+1}\sum_{d|l}\mu(d)/d+C_pl^p\sum_{d|l}\mu(d)+l^{p-1}[\sum_{d|l}\mu(d)O(1/d)],$$
Then we have
$$|\sum_{d|l}\mu(d)O(1/d)|\leq C\ln{d},$$
and by Lemma \ref{mob}:
$$\sum_{d|l}\mu(d)=0 \quad \text{if} \quad l\neq 1.$$
We deduce
$$S_l=C_{p+1}l^{p+1}\sum_{d|l}\mu(d)/d+O(l^{p-1}\ln{l}).$$

$$\sum_{l\leq n}S_l=\sum_{l\leq n}Cl^{p+1}\sum_{d|l}\mu(d)/d+\sum_{l\leq n}O(l^{p-1}\ln{l}),$$
$$\sum_{l\leq n}S_l=\sum_{d\leq n}C\mu(d)/d\sum_{n'\leq
n/d}n'^{p+1}d^{p+1}+O(n^{p}\ln{n}),$$
$$\sum_{l\leq n}S_l= C\sum_{d\leq
n}\mu(d)d^pE(\frac{n}{d})^{p+2}+O(n^{p}\ln{l}),$$
$$\sum_{l\leq n}S_l= C\sum_{d\leq
n}\mu(d)d^p[\frac{n}{d}-\{\frac{n}{d}\}]^{p+2}+O(n^{p}\ln{l}),$$
$$\sum_{l\leq n}S_l= C\sum_{d\leq
n}\mu(d)d^p[(\frac{n}{d})^{p+2}+
O((\frac{n}{d})^{p+1})]+O(n^{p}\ln{l}),$$

$$\sum_{l\leq n}S_l= Cn^{p+2}\sum_{d\leq
n}\mu(d)/d^2 +Cn^{p+1}\sum_{d\leq n}\mu(d)/d +O(n^{p}\ln{l}),$$
$$\sum_{l\leq n}S_l= Cn^{p+2}\sum_{d\leq
n}\mu(d)/d^2 +n^{p+1}O(\log{n}) +O(n^{p}\ln{l}).$$ The series of
general term $\mu(d)/d^2$ is absolutely convergent and its sum is
$\frac{1}{\zeta(2)}$ which is positive \cite{Ha.Wr}, thus we
deduce:
$$\sum_{l\leq n}S_l\geq Cn^{p+2}.$$
\end{proof}
\section{Preliminary results}
\begin{definition}\label{defdiag}
A diagonal $\gamma_{A,B}$ between two faces $A,B$ of the cubic tesselation, of dimensions less than $d-2$,
is the set of (oriented) segments which start from $A$ and stop
in $B$.
\end{definition}
\begin{definition}
We introduce the following order on the faces \!: two faces $A$
and $B$ verify $A<B$ if and only if \!: each oriented segment,  from $A$ to $B$, is such that in the unfolding, the associated
vector has positive coefficients. The diagonals are of several
types due to the dimension of $A,B$. We call a diagonal between
the faces $A,B$ a positive diagonal if we have $B>A$. If we attach
a superscript $+$ to an object, then it will consist of positive
diagonals.
\end{definition}

\begin{definition}
 We say that two faces $A,B$ are at combinatorial length
$n$ if each orbit segment between $A,B$ passes through $n$ cubes.
We denote the length by $d(A,B)=n$.
\end{definition}
This definition can be made since we are in the hypercube. In
other polyhedron it is not well defined.
\begin{definition}
If the faces $A,B$ of the cubic tesselation fulfill $d(A,B)=n$ then the diagonal
$\gamma_{A,B}$ is of combinatorial length $n$.
 We denote the set of these diagonals by $Diag(n)$.
\end{definition}

\begin{figure}[h]
\includegraphics[width= 3.5cm]{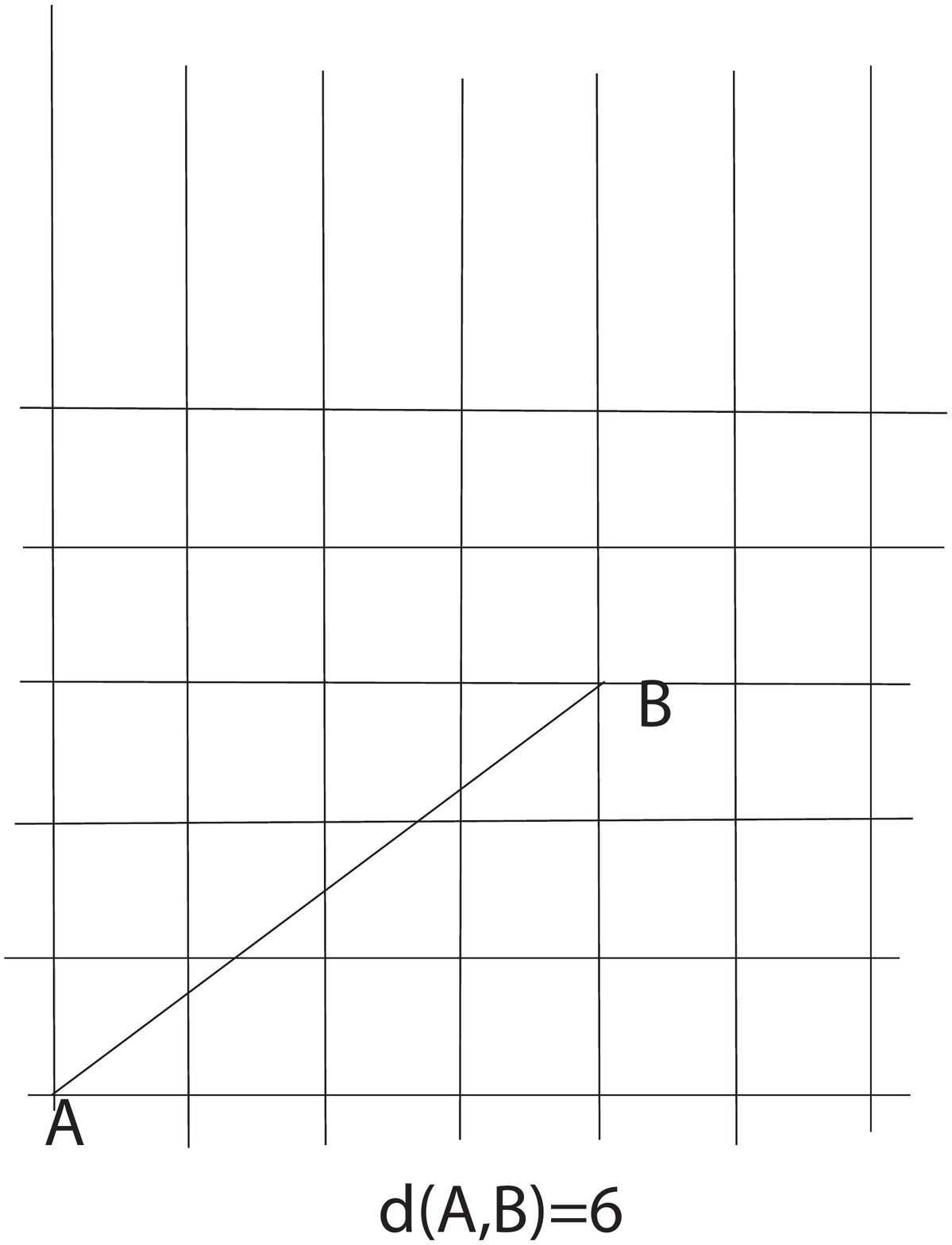}
\caption{Words of billiard}
\end{figure}

{\bf Notations.} In the following we only consider diagonals of
combinatorial length $n$ whose initial segment is in the cube
$[0,1]^d$. If a diagonal is a positive diagonal, it implies that
the final edge is in $\mathbb{R_+}^d$.

We denote the fact that an orbit in the diagonal $\gamma$ has code
$v$ by $v\in\gamma$. We consider the bispecial words such that, in
the unfolding, the associated trajectories are in
$\mathbb{R_+}^{d}$, and not in one of the $d$ coordinates planes,
we denote these words by $\mathcal{BL}(n,d)^+$.

In the following we call octant a proper subspace of
$\mathbb{R}^d$ of the form $I_1\times\dots\times I_d$ where $I_i$
is equal to $\mathbb{R}^-$ or to $\mathbb{R}^+$.
\begin{definition}
Let $v$ be a billiard word, we define the cell of $v$, by the
subset $\{(m,\omega)\in\partial{P}\times \mathbb{R}P^{d-1}\}$ such
that for all $0\leq i\leq |v|-1$, $T^i(m,\omega)\cap
\partial{P}$ is in the face labelled by $v_i$.
\end{definition}

The aim of this section is to show:
\begin{proposition}\label{pr}
With the preceding notations, there exists $C>0$ such that
\begin{equation*}
\frac{1}{2^d}\sum_{v\in\mathcal{BL}(n,d)}{i(v)}=
\sum_{v'\in\mathcal{BL}^{+}(n,d)}{i(v')}+O(|s(n+1,d-1)-s(n,d-1)|),
\end{equation*}

$$\sum_{\gamma\in
Diag(n)}\sum_{v\in\gamma}1\leq\sum_{\mathcal{BL}^{+}(n,d)}{i(v')}\leq
C\sum_{\gamma\in Diag(n)}\sum_{v\in\gamma}1,$$ where
$s(n,d)=p(n+1,d)-p(n,d)$.
\end{proposition}
For the proof we need the following lemmas.

\begin{lemma}\label{unicode}
We consider a word $v$ in $\mathcal{L}(n,d)$ with $n\geq 2$,
consider the unfolding of the billiard trajectories which are
coded by $v$ and start inside the cube $[0;1]^d$. Then for all $i,
2\leq i\leq n$, there exists only one face corresponding to the
letter $v_i$.
\end{lemma}
\begin{proof}
First we consider the intersection of the cell of $v$ with
$\mathbb{R}P^{d-1}$. This set is a proper subset of an octant since
$n\geq 2$. Now we make the proof by contradiction. We consider the
first times $j$ where two different faces appear. There exist two
lines starting form a face (corresponding to $v_{j-1}$) which pass
through these two different faces. These faces are different but
are coded by the same letter, thus they are in two different
hypercubes. Thus the two directions are in different octant,
contradiction.
\end{proof}
\begin{figure}
\includegraphics[width= 4.2cm]{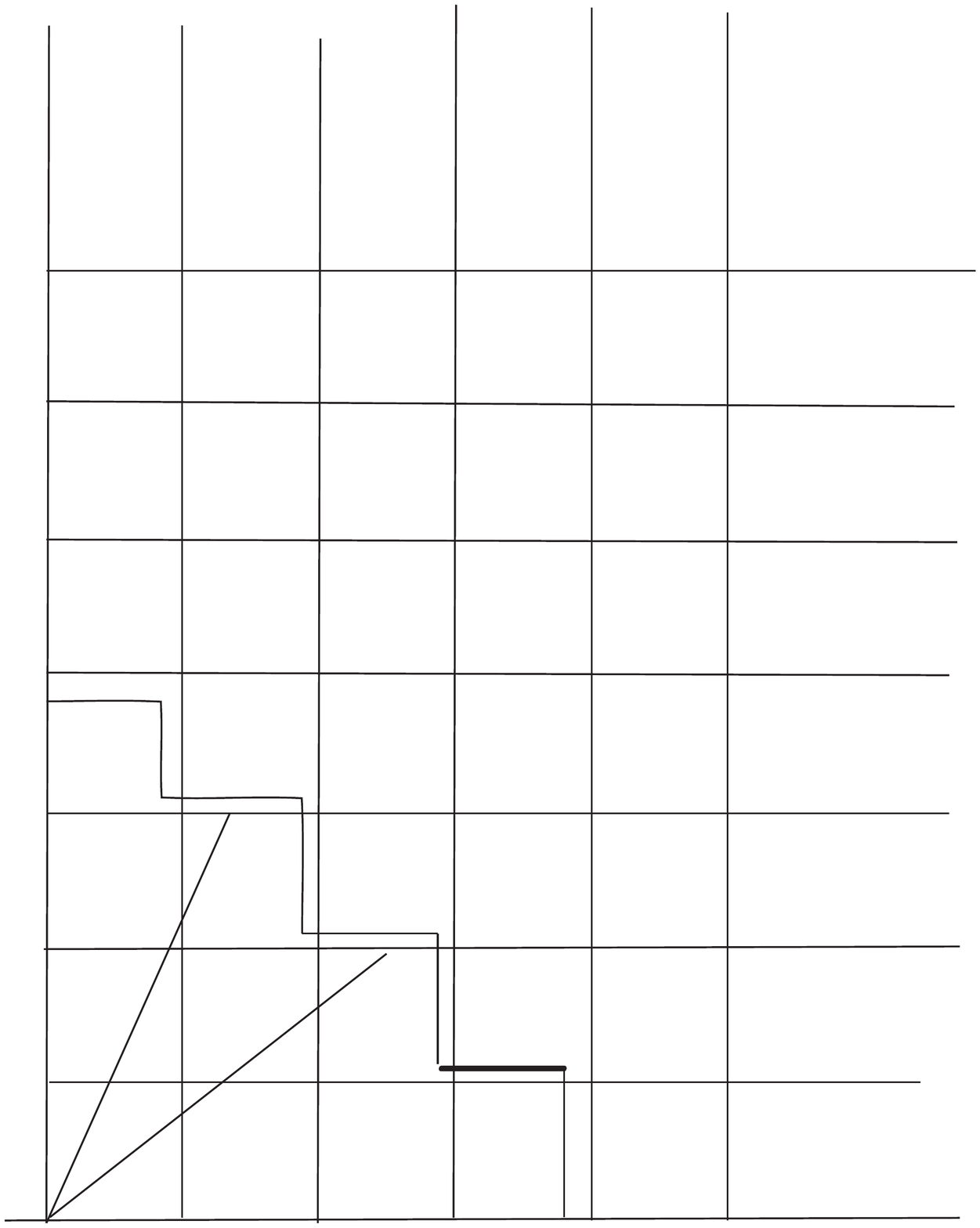}
\caption{Words of billiard}
\end{figure}
In this figure we show two billiard words in the square. The path
represents the faces at length $n$ of the initial square. In the
figure we have $n=3.$ The two words are coded by $001$ and $101$,
if we code the horizontal lines by $0$, and the vertical lines by
$1$.

\begin{lemma}\label{surjcu}
Let $v$ be a word in $\mathcal{BL}(n,d)^+$, then there exists
only one positive diagonal associated to this word.
\end{lemma}
\begin{proof}
If we study the unfolding of a trajectory associated to $v$, the
fact that we consider only words in $\mathcal{BL}(n,d)^+$ (and not
in $\mathcal{BL}(n,d)$) implies that there are at most $d$ choices
for the suffix of $v$ in the octant $\mathbb{R}^d_+$ (a suffix is
a letter $l$ such that $vl$ is a word), and the same result for
the prefix.

We consider the faces related to the suffix letter. We claim that
these faces have a non-empty intersection: By Lemma \ref{unicode}
these faces are in a same hypercube. They correspond to different
letters of the coding, thus these faces intersect (by definition
of the coding). The claim is proved.

 Those faces have a non-empty intersection, if we consider the same intersection with
the prefix, we have built a diagonal associated to this word, and
by construction it is unique.
\end{proof}

\begin{definition}
We call discontinuity a set of points $X=\{(m,\omega), m\in A\}$
in the phase space such that $A$ is a face, and such that their
orbits intersect another face of dimension $d-2$.
\end{definition}

Let us remark that a discontinuity is of dimension at most
$2(d-1)$, and that a diagonal is in the intersection of two
discontinuities.

\begin{lemma}\label{i(v)}
Consider a diagonal $\gamma$ between two faces $A,B$ of dimensions
$i,j$. Then for all word $v\in\gamma$, see Notations, we have
$$d^2\geq i(v)\geq 1.$$
\end{lemma}
\begin{proof}
Consider a bispecial billiard word $v$, the cell of $v$ in the
phase space is an open set. It means that if a trajectory has $v$
for coding, a small perturbation of $v$ has still $v$ for coding.

 A face of dimension $d-2$ is at the intersection of two faces of dimension
$d-1$, thus the face of dimension $j$ is at the intersection of at
least $d-j$ faces, and we deduce $m_r(v)\geq d-j$ by perturbation.
The same method shows that $m_l(v)\geq d-i$. Now by definition of
diagonal, see Definition \ref{defdiag}, the diagonal is in the interior of
the cell of $v$ in the phase space. Moreover the cell of $v$ is an
open set. Now consider a segment $[a;b]$ inside the diagonal with
$a\in A$. There exists an open set near $a$ such that for all $a'$
inside the segment $[a';b]$ is still coded by $v$. Now there
exists a neighborhood of $b$ such that for all $b'$ the segment
$[a';b']$ has $v$ for coding. This implies $m_b(v)=m_r(v)m_l(v)$.
Finally we obtain $i(v)=(m_l(v)-1)(m_r(v)-1)\geq 1.$ The other
inequality is obvious.
\end{proof}

\subsection{Proof of Proposition \ref{pr}}
First we remark that the symmetries of $\mathbb{R}^d$ implies that
$\sum_{v\in\mathcal{BL}(n,d)}{i(v)}$ is the same for each octant.

Now we are interested in the bispecial words which are neither in
$\mathcal{BL}(n,d)^+$ nor in one of the symmetric sets. Their
unfolding is in $[0,1]^{d-1}\times\mathbb{R}^+$. Thus for each
coordinates plane their number is equal to the number of bispecial
words of the cube of dimension $d-1$. Lemma \ref{julien} implies:
$$\displaystyle\sum_{v\notin \mathcal{BL}(n,d)^+}i(v)\leq
C|s(n+1,d-1)-s(n,d-1)|\quad C\in\mathbb{R}.$$

We consider the map $$f\!:\mathcal{BL}(n,d)^+ \rightarrow
Diag^+(n)$$
$$f:v\mapsto\gamma$$

Lemma \ref{surjcu} implies that $f$ is well defined and onto, thus
$$card(\mathcal{BL}(n,d)^+)=\sum_{\gamma \in Diag(n)}{card(f^{-1}(\gamma))}.$$

Then we obtain
$$\sum_{\mathcal{BL}(n,d)^+}{i(v')}=\sum_{\gamma \in Diag(n)}\sum_{v\in\gamma}i(v).$$

Now we must bound $i(v)$ for each $v\in\gamma$. This is a
consequence of Lemma \ref{i(v)}.

\section{Equations of diagonals}
In these section we give in Lemma \ref{cal} the equations of a
diagonal, we deduce in Proposition \ref{calcul} that several
diagonals can not overlay, and we finish the section by a
description of the diagonals of fixed combinatorial length. Remark
that these equations are homogeneous in $\omega$.
\begin{lemma}\label{cal}
Let $A,B$ two faces of dimension $d-2$, we consider
$$\gamma_{A,B}=\{(m,\omega)\in\mathbb{R}^{d}\times\mathbb{R^*}^{d}, m\in
A, m+\mathbb{R}\omega\cap B\neq\emptyset\}.$$ Then $\gamma_{A,B}$
has one
of the following equation\\
$(1)\bullet n\omega_i=p\omega_j$, with $n,p\in\mathbb{N}$.\\
$(2)\bullet m_i+\frac{n\omega_i}{\omega_j}=p$ with $n,p\in\mathbb{N}$.\\
$(3)\bullet \omega_jm_i-\omega_im_j=n\omega_i+p\omega_j$ with
$n,p\in\mathbb{N}$.
\end{lemma}
\begin{proof}
First we can assume that the point $m\in A$ have coordinates of
the following form
$$\begin{pmatrix}
m_1\\ \vdots\\ m_{d-2}\\ 0\\0
\end{pmatrix}.$$
Then each point of $B$ have two coordinates equal to integers
$n,p$. Thus its coordinates are of the form:
$$B:\begin{pmatrix}
b_1\\ \vdots\\n\\ \vdots\\ p\\ \vdots\\ b_{d-2}\end{pmatrix}.$$ If
the line $m+\mathbb{R}\omega$ intersects $B$ it means that there
exists $\lambda$ such that $m+\lambda\omega\in B$. Then there are
three choices, depending on the position of $n,p$ in the
coordinates. \\
$\bullet$ If $n,p$ are at positions $d-1,d$ we obtain a system of
the form
$$\begin{cases}
\lambda\omega_{d-1}=n\\ \lambda\omega_d=p
\end{cases}$$
This gives equation $(1)$.\\

$\bullet$If $n$ is at a position $i$ less or equal than $d-2$, and
$p$ is at position $d-1$ or $d$, we obtain
$$\begin{cases}
\lambda\omega_{d-1}=p\\
m_{i}+\lambda\omega_{i}=n
\end{cases}$$
This gives the second equation.\\

$\bullet$ If $n$ and $p$ are at position less than $d-2$, we are
in case (3).
\end{proof}
\begin{proposition}\label{calcul}
Let $A,B,C_i, i=1\dots l$ be $l+2$ faces of dimension $d-2$. We
deduce the equivalence
$$\gamma_{A,B}=\bigcup_{i}\gamma_{A,C_i}\Longleftrightarrow A,B,C_i\ \text{are contained in a
hyperplane of}\quad \mathbb{R}^d.$$
\end{proposition}
\begin{proof}
We consider the three functions which appear in Lemma \ref{cal}.
$$\begin{cases}
f(\omega)=n\omega_i-p\omega_j,\\
g(m,\omega)=m_i+\frac{n\omega_i}{\omega_j}=p,\\
h(m,\omega)=\omega_jm_i-\omega_im_j-(n\omega_i+p\omega_j).
\end{cases}$$
The diagonals $\gamma_{A,B},\gamma_{A,C_i}$ have equations of the
type $f,g,h$ by preceding Lemma (with different $n,p,i,j$). Remark
that these equations are quadratic in the variables $m,\omega$.
Thus these maps are analytic.

We compute the jacobian of these maps, it gives
$$
\begin{cases}
df=\begin{pmatrix}\dots&0&\dots||&\dots&n&\dots -p\dots
\end{pmatrix}\\
dg=\begin{pmatrix}0&\dots&\omega_j&\dots&||&\dots&n&\dots
m_i-p\dots
\end{pmatrix}\\
dh=\begin{pmatrix}0&\dots&\omega_j&\dots&-\omega_i&\dots||&\dots&-m_i-n&\dots
m_i-p\dots
\end{pmatrix}
\end{cases}$$
 Now without loss of generality we treat the case $l=1$.
 The sets $\gamma_{A,B}$, $\gamma_{A,C}$ are equal if and
only if two of the preceding functions are equal on a set of
positive measure. It implies that the linear forms are
proportional. Assume that two different forms are proportional
(for example $df$ and $dg$). It implies that $m_i=0$, thus the
equality is true on an hyperplane, and they are not equal on a set
of positive measure. Thus the only possibility is that the two
equations are of the same type (i.e two equations $df$ or two
equations $dg$). Then the same argument shows that the equality of
two equations of the type $dg$ or $dh$ implies that $(m,\omega)$
lives on a set of zero measure. Thus the only possibility is the
equality of two vectors $df$. And it is equivalent to the fact
that $A,B,C$ belong to the same hyperplane.
\end{proof}
\begin{lemma}\label{long}
Let $A,B$ be two faces of dimension less or equal than $d-2$. Assume
$A,B$ are at combinatorial length $n$, and that the elements of
$A$ are of the form
$$\begin{pmatrix}
m_1\\ \vdots\\ m_{d-2}\\ 0\\0
\end{pmatrix}.$$
Then we have:\\
$\bullet$Either $A,B$ are in a subspace of dimension $d-2$ then
points of $B$ are
of the form $$\begin{pmatrix} b_1\\ \vdots\\ b_{d-2}\\
n_{d-1}\\n_d
\end{pmatrix},$$ with
$n_{d-1},n_d\in\mathbb{N}$, $\gcd(n_{d-1},n_d)=1\ and\ \displaystyle\sum_{i=1}^{d-2}E(b_i)+n_{d-1}+n_d=n$\\

$\bullet$ Or the points of $B$ have the following coordinates:
$$(i,j)\neq (d-1,d)\quad \begin{pmatrix}
b_1\\ \vdots\\ n_i\\ \vdots\\ n_j\\ \vdots\\
b_{d-2}\end{pmatrix},$$ with $n_{i},n_j\in\mathbb{N}\ and\
\displaystyle\sum_{l=1}^{d-2}E(b_l)+n_{i}+n_j=n$.
\end{lemma}
\begin{proof}
$\bullet$ First of all we consider the faces of dimension $d-1$
which are at combinatorial length $n$ of $A$. We claim that the
points $(b_i)_{i\leq d}$ of these faces verify
$\displaystyle\sum_{i=1}^{d}E(b_i)=n$.\\
The proof is made by induction on $n$. It is clear for $n=1$, now
consider a billiard trajectory of length $n$, it means that just
before the last face we intersect another face of the same cube.
These face is at combinatorial length $n-1$, and we can apply the
induction process. Now consider a point of these faces, denote by
$(c_i)_{i\leq d}$ its coordinates. We verify easily that
$\displaystyle\sum_{i=1}^{d-2}E(b_i)-\displaystyle\sum_{i=1}^{d}E(c_i)=1$
for all point $b,c$. This finishes the proof of the claim.
\begin{figure}
\includegraphics[width= 4.2cm]{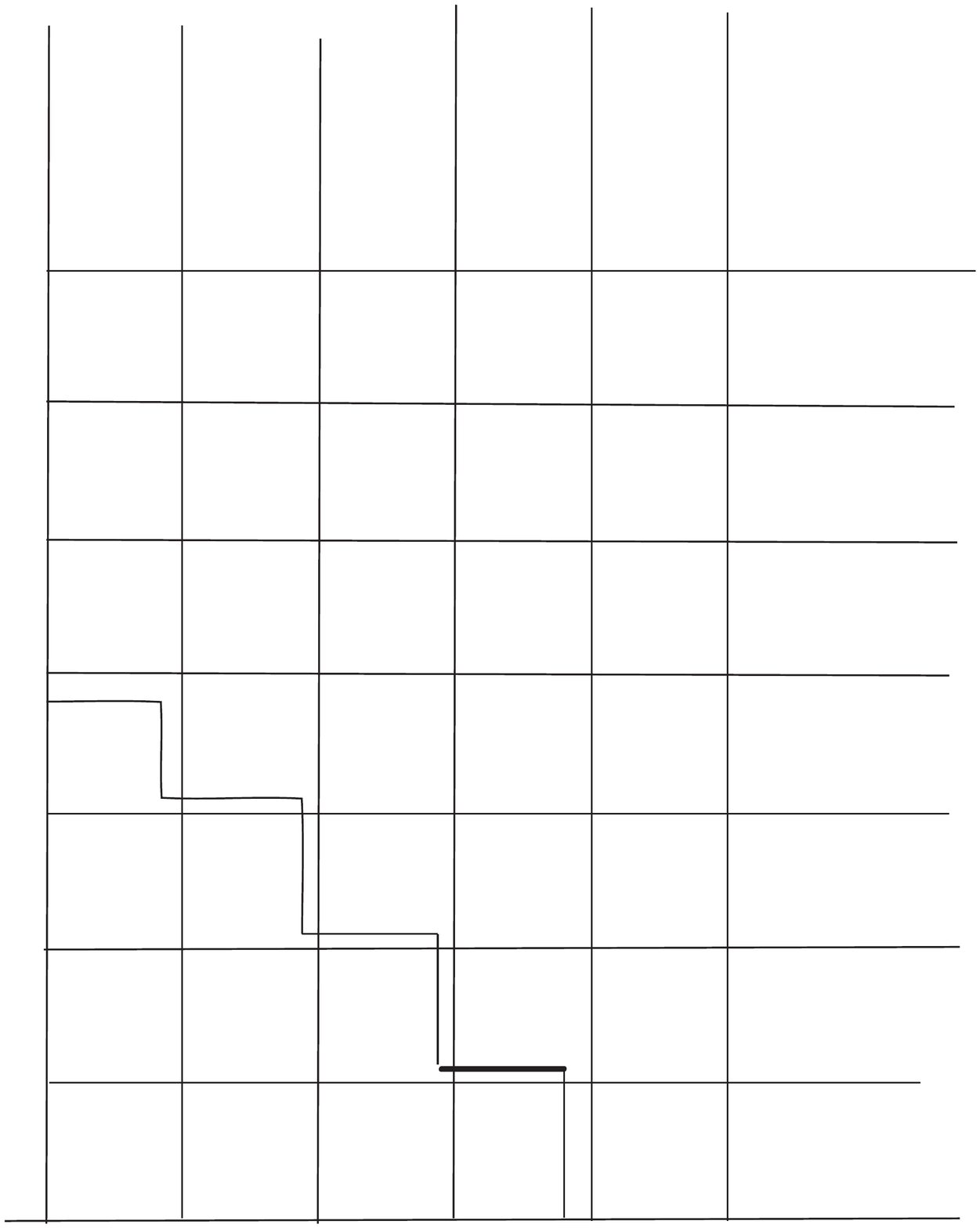}
\caption{Length of billiard words}
\end{figure}
In this figure the path represents the faces at length $n$ of the
initial square. In the figure we have $n=3.$

 A trajectory between $A$ and $B$ is a diagonal if the trajectory does not
intersect another face. It means that $\gamma_{A,B}$ must not be
the union of $\gamma_{A,C_i}$, where $C_i$ are at length less than
$n$ from $A$. We use Proposition \ref{calcul} which implies that
the only bad case is when $A,B$ are on a
same hyperplane. Thus the second point of the Lemma is proved.\\

Now assume $A,B$ are contained in a hyperplane. The fixed
coordinates of all points in $A$ and $B$ are at the same places.
Then we project on the plane generated by these coordinates. The
diagonal projects on a line. This line does not contain integer points. Thus we obtain the primality condition.\\

\end{proof}
\begin{corollary}\label{diag}
We deduce that there exists $C>0$ such that
$$card Diag(n)\leq Cn^{d-1}.$$
\end{corollary}
\begin{proof}
A diagonal $\gamma_{A,B}$ can be of several forms among the
dimension of the faces. Since a face of dimension $d-1$ has a
bounded number of faces of dimension less than $d-1$ in its
boundary, we can reduce to count the diagonals between faces of
dimension $d-1$. Then the number of diagonals is bounded by a
constant $C(d)$ times the number of diagonals between faces of
dimension $d-2$. The preceding Lemma shows that we have the
inequality
$$card Diag(n)\leq C card\{(n_i)_{1\leq i\leq d}, n_i\in\mathbb{N}|
\displaystyle\sum_{i=1}^{d}n_i=n\}.$$
$$card Diag(n)\leq Cn^{d-1}.$$
\end{proof}
\section{Upper bound}
In this section we show
\begin{theorem}\label{min}
There exists $C>0$ such that
$$s(n,d)\leq Cn^{3d-4}.$$
\end{theorem}
\begin{lemma}\label{connexe}
Let $A,B$ two faces of dimension less than $d-2$, then the set
$\gamma_{A,B}$ is of dimension less than $2(d-2)$.\\
For all $k$, for all subset $I$ of $\mathbb{N}$ of cardinality
$k$, there exists $C(k)>0$ such that for all $(A_i)_{i\in I}$
faces of dimension less than $d-2$, $\bigcap_{i,j\in
I}\gamma_{A_i,A_j}$ has at most $C(k,d)$ connected components.
\end{lemma}
\begin{proof}
The first part is a consequence of lemma \ref{cal}. Indeed $A$ is
of dimension less than $d-2$, the directions lives in
$\mathbb{R}P^{d-1}$ which is of dimension $d-1$ and the manifold has one equation.\\
For the second part we use again Lemma \ref{cal}. The equation of
these sets are polynomial equation of bounded degree (2), and a
theorem of \cite{Ful} (Ex 8.4.5) finishes the proof.
\end{proof}

\begin{proposition}
There exists $C>0$ such that for all $A,B$ faces of dimension less
than $d-2$, at combinatorial length $n$,  we have
$$\sum_{v\in\gamma_{A,B}}1\leq Cn^{2d-4}.$$
\end{proposition}
\begin{proof}
We consider the cell related to $\gamma_{A,B}$. This space is
partitioned with several discontinuities. The number of sets of
the partition is equal to the number of words $v$ in $\gamma_{A,B}$. First if a
discontinuity does not partition, we prolong it. It gives an upper
bound for the number of words. Then we consider the partition made
by two discontinuities, Lemma \ref{connexe} implies that the
number of connected components is bounded by $C$. Then we apply
Corollary \ref{comb} with $Cn$ hyperplanes (in fact algebraic varieties of degree at most 2 see Remark \ref{rem}), and $x=2(d-2)$ due to
the first part of Lemma \ref{connexe}.
\end{proof}
\subsection{Proof of Theorem \ref{min}}
We make an induction on $d$. If $d=2$ it is a consequence of
\cite{Mig} or \cite{Ber.Poc}, or \cite{Ca.Hu.Tr}.

 By Proposition
\ref{pr} we deduce
$$\displaystyle\sum_{v\in\mathcal{BL}(n,d)}i(v)\leq
2^d[\sum_{\gamma\in Diag(n)}\sum_{v\in\gamma}1+s(n+1,d-1)].$$ Then
the preceding proposition shows
$$\displaystyle\sum_{v\in\mathcal{BL}(n,d)}i(v)\leq
2^d[\sum_{\gamma\in Diag(n)}Cn^{2d-4}+s(n+1,d-1)].$$ 
Corollary \ref{diag}
implies that $card(Diag(n))\leq n^{d-1}$, we deduce
$$\displaystyle\sum_{v\in\mathcal{BL}(n,d)}i(v)\leq
2^d[Cn^{3d-5}+s(n+1,d-1)],$$

By induction we deduce
$$s(n+1,d)-s(n,d)\leq C[n^{3d-5}+n^{3d-7}].$$
$$s(n,d)\leq Cn^{3d-4}.$$
The induction is proved.

\section{Lower bound} We prove
\begin{theorem}\label{maj}
There exists $C>0$ such that for all $n$:
$$s(n+1,d)-s(n,d)\geq Cn^{3d-5}.$$
\end{theorem}
The proof is made by induction on $d$. It is clear for $d=2$ due
to \cite{Mig} or \cite{Ber.Poc} or \cite{Ca.Hu.Tr}, assume it is
true for $i\leq d-1$.

\begin{definition}
Let $A,B$ two faces of dimension less than $d-2$, we denote by
$C_{A,B}$ the vector space generated by $\vec{A},\vec{B}$.\\
Let $\pi:\mathbb{R}^d\rightarrow C_{A,B}$ be the orthogonal
projection.
\end{definition}

 Consider a trajectory of a fixed diagonal, it is
coded by a word $v$, the image of the trajectory by $\pi$ is a
billiard trajectory, due to Lemma \ref{proj}. Thus the map $\pi$
can be extended to words, we denote it again by $\pi$.
\begin{equation*}
\begin{CD}
\gamma   @>\pi>>  \pi(\gamma)\\
@V{\phi}VV        @VV{\phi}V\\
v @>>\pi> \pi(v)
\end{CD}
\end{equation*}
The map $\pi$ consists to erase some letters, due to Lemma
\ref{proj}.

\subsection{Projection and language}
The aim of this section is to prove
\begin{proposition}\label{projec}
We have
$$\displaystyle\bigcup_{\gamma_{A,B}\in Diag(n)}\{\pi(v),
v\in\gamma_{A,B}\}=\displaystyle\bigcup_{i\leq n-1}\mathcal{L}
(i,d-1).$$
\end{proposition}
\begin{proof}
The inclusion $\displaystyle\bigcup_{\gamma_{A,B}\in Diag(n)}\{\pi(v),
v\in\gamma_{A,B}\}\subset\displaystyle\bigcup_{i\leq n-1}\mathcal{L}
(i,d-1).$

is a consequence of Lemma \ref{proj}. To prove
the second inclusion we need:
\begin{lemma}
Let $i\leq n-1$, and let $v\in\mathcal{L}(i,d-1)$ be a billiard
word between two faces $A,B'$ of dimension $d-2$. There exists a
face $B$ of dimension $d-2$ such that:\\
$$d(A,B)=n,$$
$$\gamma_{A,B}\ \text{is a diagonal},$$
$$\pi_{A,B}(B)=B'.$$
\end{lemma}
\begin{proof}
By Lemma \ref{long}, we can always lift the face $B'$ in a face
$B$ with  $d(A,B)=n$. We just have to translate $B'$ to the
coordinate $x_d=n-i$. The only point to prove is that the
trajectories between $A,B$ form a diagonal. We make a proof by
contradiction. Then each trajectory between $A,B$ intersects
another face $C_i$. It implies that $\gamma_{A,B}$ is cover by
some $\gamma_{A,C_i}$. Contradiction with Proposition
\ref{calcul}.
\end{proof}
Now the proof of the Proposition is a simple consequence of this
lemma and of Lemma \ref{proj}.
\end{proof}
\begin{corollary}\label{proj2}
For any diagonal $\gamma$ we have
$$\sum_{v\in\gamma}1\geq \sum_{v\in\pi(\gamma)}1.$$
\end{corollary}
\begin{proof}
By preceding Lemma, a word of $\pi(\gamma)$ can be lift in a word
of $\gamma$. In other word the map $\pi$ is surjective on the
billiard words.
\end{proof}

\subsection{Proof of Theorem \ref{maj}}
 We fix the face $A$ as
in Lemma \ref{cal}. By Lemma \ref{long} the coordinates
$(n_1,\dots,n_d)$ of $B$ can be of two types:
$$n_1+\dots+n_d=n,\ \gcd(n_{d-1},n_d)=1\quad\text{or};$$
$$n_1+\dots+n_d=n.$$
\begin{definition}
We denote these sets of diagonals by $Diag_1(n)$ and $Diag_2(n)$.
Let $\gamma$ be a diagonal, the number $\sum_{v\in \gamma}1$ is
denoted by $f(n_1,\dots,n_d)$ or $g(n_1,\dots,n_d)$ if $\gamma$ is
in $Diag_1(n)$ or not.
\end{definition}
Due to Proposition \ref{pr} we must compute
$$X=\displaystyle\sum_{\gamma\in Diag(n)}\sum_{v\in \gamma}1.$$
By Corollary \ref{proj2} we can write this sum as
$$X=\displaystyle\sum_{\gamma\in Diag_1(n)}\sum_{v\in \gamma}1+\displaystyle\sum_{\gamma\in Diag_2(n)}\sum_{v\in \gamma}1.$$
$$X=\displaystyle\sum_{\gamma\in Diag_1(n)}f(n_1,\dots,n_d)+\sum_{\gamma\in Diag_2(n)}g(n_1,\dots,n_d).$$
Now we use the projection $\pi$:

 By Lemma \ref{long} and \ref{proj2} we deduce
$$X\geq\displaystyle\sum_{n_d,n_{d-1}}\sum_{n_1,\dots,n_{d-2}}[f(n_1,\dots,n_{d-2})\chi(n_d,n_{d-1})+g(n_1,\dots,n_d)],$$
where $\chi(n_d,n_{d-1})=\begin{cases}1\ if\ \gcd(n_d,n_{d-1})=1\\
0\ either \end{cases}$

Now Proposition \ref{projec} implies that
$$\displaystyle\sum_{n_1+\dots+n_{d-1}=n}[f(n_1,\dots,n_{d-1})+g(n_1,\dots,n_{d-1})]=p(n,d-1).$$
This can be written as
$$\displaystyle\sum_{n_{d-1}\leq n}\sum_{n_1+\dots+n_{d-2}=n-n_{d-1}}[f(n_1,\dots,n_{d-1})+g(n_1,\dots,n_{d-1})]=p(n,d-1).$$

$$\displaystyle\sum_{n_1+\dots+n_{d-2}=n-n_{d-1}}[f(n_1,\dots,n_{d-1})+g(n_1,\dots,n_{d-1})]=s(n_{d-1},d-1).$$
 Then we deduce
$$X\geq\displaystyle\sum_{n_d,n_{d-1}}\sum_{n_1,\dots,n_{d-2}}[f(n_1,\dots,n_{d-2})+g(n_1,\dots,n_d)]\chi(n_{d-1},n_d),$$
$$X\geq\displaystyle\sum_{n_d}\sum_{n_{d-1}}\sum_{n_1,\dots,n_{d-2}}[f(n_1,\dots,n_{d-2})+g(n_1,\dots,n_d)]\chi(n_{d-1},n_d),$$
$$X\geq \displaystyle\sum_{n_d\leq n}\sum_{n_{d-1}\leq n-n_d}s(n_d{-1},d-1)\chi(n_d,n_{d-1}).$$

Then the induction hypothesis shows that
$$|s(n+1,d-1)-s(n,d-1)|\geq n^{3d-8}.$$
It implies a lower bound on $s(n,d-1)$.
$$X\geq \displaystyle\sum_{n_d\leq n}\sum_{n_{d-1}\leq
n-n_d}(n_{d-1})^{3d-7}\chi(n_d,n_{d-1}).$$
$$X\geq \displaystyle\sum_{n_d\leq n}\sum_{n_{d-1}\leq
n_d}(n_{d-1})^{3d-7}\chi(n-n_d,n_{d-1}).$$
$$X\geq \displaystyle\sum_{n_d\leq n}\sum_{\substack{n_{d-1}\leq
n_d\\\gcd(n_{d-1},n_d)=1}}(n_{d-1})^{3d-7}.$$

We apply Lemma \ref{nombre} with $p=3d-7$ and obtain:
$$X\geq Cn^{3d-5}.$$
Now by Proposition \ref{pr} we have
$$s(n+1,d)-s(n,d)\geq CX+O(s(n,d-1)-s(n,d-1)).$$

 We apply again the induction
to obtain a lower bound for the error term of Proposition
\ref{pr}. This term is bounded by $n^{3d-8}$. Thus we have
$$s(n+1,d)-s(n,d)\geq Cn^{3d-5}-n^{3d-8}.$$
$$s(n+1,d)-s(n)\geq Cn^{3d-5}.$$
The proof by induction is finished.
\section{Proof of the main theorem}
We just have to join Theorems \ref{min} and \ref{maj}. $\Box$
\\

{\bf Acknowledgment} We would like to thank Christian Mauduit, who
worked on an earlier version of this paper. Moreover we would like to thank Joel Rivat for useful discussions.

\bibliographystyle{alpha}
\bibliography{biblioc}
\end{document}